\documentclass[11pt]{article}

\usepackage{amsmath}
\usepackage{amsthm}
\usepackage{amsfonts}
\usepackage{bbm}
\usepackage{amssymb}
\usepackage{color}

\usepackage{xcolor}

\newcommand{\mmp}{\mathbb{P}}

\newcommand{\me}{\mathbb{E}}
\newcommand{\E}{\mathbb{E}}

\newcommand{\mr}{\mathbb{R}}
\newcommand{\mn}{\mathbb{N}}

\newcommand{\mm}{\mathcal{M}}

\DeclareMathOperator{\1}{\mathbbm{1}}
\DeclareMathOperator{\V}{\mathbb{V}}
\DeclareMathOperator{\D}{\mathcal{D}}
\newcommand{\R}{\mathbb{R}}
\DeclareMathOperator{\F}{\mathcal{F}}
\DeclareMathOperator{\N}{\mathbb{N}}
\DeclareMathOperator{\C}{\mathbb{C}}

\newtheorem{thm}{Theorem}[section]
\newtheorem{lemma}[thm]{Lemma}

\newtheorem{assertion}[thm]{Proposition}
\theoremstyle{definition}

\theoremstyle{remark}
\newtheorem{rem}[thm]{Remark}

\begin{document}
\title{On $L^p$-convergence of the Biggins martingale with complex parameter}

\author{Alexander Iksanov\footnote{Faculty of Computer Science and Cybernetics,
Taras Shevchenko National University of Kyiv, 01601 Kyiv, Ukraine;
\ e-mail: iksan@univ.kiev.ua}, \ Xingang Liang\footnote{School of
Science, Beijing Technology and Business University, 100048
Beijing, China; \ e-mail: liangxingang@th.btbu.edu.cn} \ and \
Quansheng Liu\footnote{Laboratoire de Math\'{e}matiques de
Bretagne Atlantique, UMR 6205, Universit\'{e} de Bretagne-Sud, F-56017 Vannes, France; \ e-mail:
quansheng.liu@univ-ubs.fr}}

\maketitle
\begin{abstract}
\noindent We prove necessary and sufficient conditions for the
$L^p$-convergence, $p>1$, of the Biggins martingale with complex
parameter in the supercritical branching random walk. The results
and their proofs are much more involved (especially in the case
$p\in (1,2)$) than those for the Biggins martingale with real
parameter. Our conditions are ultimate in the case $p\geq 2$ only.
\end{abstract}
%%\noindent Key words:

\section{Introduction}

We start by recalling the definition of the branching random walk.
Consider an individual, the ancestor, located at the origin of the
real line at time $n=0$. At time $n=1$ the ancestor produces a
random number $J$ of offspring which are placed at points of the
real line according to a random point process $\mm = \sum_{i=1}^J
\delta_{X_i}$ on $\R$ with intensity measure $\mu$ (particularly,
$J = \mm(\R)$). The random variable $J$ is allowed to be infinite
with positive probability. The first generation formed by the
offspring of the ancestor produces the second generation whose
displacements with respect to their mothers are distributed
according to independent copies of the same point process $\mm$.
The second generation produces the third one, and so on. All
individuals act independently of each other.

More formally, let $\V=\cup_{n\in\mn_0}\mn^n$ be the set of all
possible individuals. The ancestor is identified with the empty
word $\varnothing$ and its position is $S(\varnothing)=0$. On some
probability space $(\Omega, \F, \mmp)$ let $(\mm(u))_{u \in \V}$
be a family of independent identically distributed (i.i.d.)\
copies of the point process $\mm$. An individual $u = u_1\ldots
u_n$ of the $n$th generation whose position on the real line is
denoted by $S(u)$ produces at time $n+1$ a random number $J(u)$ of
offspring which are placed at random locations on $\R$ given by
the positions of the point process
$$
\sum_{i=1}^{J(u)} \delta_{S(u) + X_i(u)}
$$
where $\mm(u) = \sum_{i=1}^{J(u)} \delta_{X_i(u)}$ and $J(u)$ is
the number of points in $\mm(u)$. The offspring of the individual
$u$ are enumerated by $ui = u_1 \ldots u_n i$, where
$i=1,\ldots,J(u)$ (if $J(u)<\infty$) or $i=1,2,\ldots$ (if
$J(u)=\infty$), and the positions of the offspring are denoted by
$S(ui)$.  Note that no assumptions are imposed on the dependence
structure of the random variables $J(u), X_1(u),X_2(u),\ldots$ for
fixed $u\in\V$. The point process of the positions of the $n$th
generation individuals will be denoted by $\mm_n$ so that
$\mm_0=\delta_0$ and
$$
\mm_{n+1} = \sum_{|u|=n} \sum_{i=1}^{J(u)} \delta_{S(u)+X_i(u)},
$$
where, by convention, $|u|=n$ means that the sum is taken over all
individuals of the $n$th generation rather than over all
$u\in\N^n$. The sequence of point processes $(\mm_n)_{n \in
\mn_0}$ is then called a \emph{branching random walk} (BRW).

Throughout the paper, we assume that the BRW is
\emph{supercritical}, that is $\E J>1$. In this case, the event
$\mathcal{S}$ that the population survives has positive
probability. Note that, provided that $J<\infty$ almost surely
(a.s.),\ the sequence $(\mm_n(\mr))_{n\in\mn_0}$ of generation
sizes in the BRW forms a Galton--Watson process.

The Laplace transform of the intensity measure $\mu$
\begin{equation*}    \label{eq:m}
m(\lambda):= \int_\R e^{-\lambda x}\mu({\rm d}x)=\me
\sum_{|u|=1}e^{-\lambda S(u)},\quad \lambda=\theta+{\rm
i}\gamma\in \mathbb{C}
\end{equation*}
plays an important role in what follows. Throughout the paper we
reserve the notation $\theta$ for the real part of $\lambda$ and
$\gamma$ for the imaginary part of $\lambda$, and assume that
$$\D:= \{\lambda\in \C: m(\lambda)~\text{converges
absolutely}\}=\{\theta\in\R: m(\theta)<\infty\}+{\rm i}\R\neq \oslash.$$
Further, we define the sets
$$\D_0:=\D\cap \{\lambda\in \C: |m(\lambda)|=0\},\quad \D_{\neq
0}=\D\backslash\D_0.$$ For $\lambda\in \D_{\neq 0}$ and
$n\in\mn_0$ set $$Z_n(\lambda):=\frac{1}{m(\lambda)^n}\int_\R
e^{-\lambda x}\mm_n({\rm
d}x)=\frac{1}{m(\lambda)^n}\sum_{|u|=n}e^{-\lambda S(u)}.$$ Let
$\F_0$ be the trivial $\sigma$-field and $\F_n$ the $\sigma$-field
generated by the first $n$ generations, that is,
$\F_n:=\sigma(\mm(u): u\in \cup_{k=0}^{n-1}\N^k)$. The sequence
$(Z_n(\lambda), \F_n)_{n\in\mn_0}$ forms a complex-valued
martingale of mean one that we call the Biggins martingale with
complex parameter. A non-exhaustive list of very recent articles
investigating these objects includes \cite{Damek+Mentemeier:2018,
Grubel+Kabluchko:2016, Iksanov+Kolesko+Meiners:2018+,
Kolesko+Meiners:2017}. We would like to stress that the Biggins
martingale with complex parameter has received much less attention
than its counterpart with real parameter and similar martingale
related to a branching Brownian motion. See
\cite{Hartung+Klimovsky:2018,Maillard+Pain:2019+} for recent
contributions in the latter case.

The purpose of this article is to provide necessary and sufficient
conditions for the $L^p$-convergence of the martingale
$(Z_n(\lambda))_{n\in\mn_0}$ for $p>1$. Our main results, Theorems
\ref{main3}, \ref{main4} and \ref{main2}, improve upon Theorem 1
in \cite{Biggins:1992} and Theorem 5.1.1 in the unpublished thesis
\cite{Liang:2010} which give sufficient conditions for the
aforementioned convergence in the cases $p\in (1,2]$ and $p>2$,
respectively. Necessary and sufficient conditions for the
$L^1$-convergence of $(Z_n(\lambda))_{n\in\mn_0}$ are beyond our
reach. Finding them seems to be a major open problem for the
Biggins martingales with complex parameter. For the time being,
our necessary and sufficient conditions for the $L^p$-convergence
for $p$ close to $1$ can be used as (non-optimal) sufficient
conditions for the $L^1$-convergence.

The rest of the paper is organized as follows. We give some
preliminaries in Section \ref{prelim}. Our results are formulated
in Section \ref{results} and then proved in Section \ref{proofs}.

\section{Preliminaries}\label{prelim}
Let $\lambda=\theta+{\rm i}\gamma \in \D$ be fixed. Keeping in
mind the inequality $|m(\lambda)|\leq m(\theta)$ we distinguish
three cases: $$ \mbox{ (I) }  |m(\lambda)|=m(\theta); \quad \mbox{
(II) }  0<|m(\lambda)|<m(\theta); \quad \mbox{ (III)
}|m(\lambda)|=0.$$

Perhaps, it is not obvious that Case III can occur. To convince
the reader we give an example of the BRW satisfying
$m(\theta)<\infty$ and $m(\lambda)=0$ for some $\lambda\in\C$. Let
$$\mu({\rm d}x)=\frac{2}{\pi}\frac{e^{\theta x}(1-\cos
x)}{x^2}\1_\R (x){\rm d}x.$$ Then $m(\theta)=2$ and
$$m(\lambda)=\frac{4}{\pi}\int_\mr e^{-{\rm i}\gamma
x}x^{-2}(1-\cos x){\rm d}x=2(1-|\gamma|)\1_{(-1,1)}(\gamma).$$ In
particular, $m(\lambda)=0$ whenever $|\gamma|>1$.

We do not touch Case III in this paper, just because the sequence
$(Z^{(0)}_n(\lambda), \F_n)_{n\in\mn_0}$ defined for $\lambda\in
\D_0$ by
$$Z^{(0)}_n(\lambda):=\int_\R e^{-\lambda x}\mm_n({\rm
d}x)=\sum_{|u|=n}e^{-\lambda S(u)}$$ does not form a martingale,
for it is comprised of complex-valued martingale differences.

\bigskip \noindent {\sc \underline{Case I}: $|m(\lambda)|=m(\theta)$.}
Since $m(\lambda)=e^{{\rm i}\varphi} m(\theta)$ for some
$\varphi\in [0,2\pi)$ we infer
$$\me \sum_{|u|=1}e^{-\theta S(u)}\Big(e^{-{\rm i}(\varphi+\gamma
S(u))}-1\Big)=0$$ and thereupon $e^{-{\rm i}\gamma S(u)}=e^{{\rm
i}(\varphi+2\pi k)}=e^{{\rm i}\varphi}$ for integer $k$ whenever
$|u|=1$. This gives $Z_n(\lambda)= Z_n(\theta)$ for $n\in\mn$ a.s.
Therefore, $(Z_n(\lambda), \F_n)_{n\in\mn_0}$ is a nonnegative
unit mean martingale.

Proposition \ref{conv_Lp} reminds a criterion for the
$L^p$-convergence ($p>1$) of the Biggins martingale with real
parameter. The result is well-known and can be found in Theorem
2.1 of \cite{Liu:2000}, Corollary 5 of \cite{Iksanov:2004},
Theorem 3.1 of \cite{Alsmeyer+Kuhlbusch:2010}, and perhaps some
other articles.

\begin{assertion}\label{conv_Lp}
Let $p>1$ and $m(\theta)<\infty$ for some $\theta\in\mr$. Then the
martingale $(Z_n(\theta))_{n\in\mn_0}$ converges in $L^p$ if, and
only if, $$\E [Z_1(\theta)]^p<\infty\quad \text{and}\quad
\frac{m(p\theta)}{m(\theta)^p}<1.$$
\end{assertion}
\begin{rem}\label{m0}
When $\theta=0$ and $m(0)<\infty$, the condition $m(0)<m(0)^p$
holds automatically because $m(0)>1$ by supercriticality. Hence,
the martingale $(Z_n(0))_{n\in\mn_0}$ converges in $L^p$ if, and
only if, $\E [Z_1(0)]^p<\infty$. This result goes back to
Corollary on p.~714 in \cite{Bingham+Doney:1974}. \end{rem}

Therefore, in Case I we conclude that the martingale
$(Z_n(\lambda))_{n\in\mn_0}$ converges in $L^p$ if, and only if,
the conditions of Proposition \ref{conv_Lp} hold true.

\bigskip \noindent \noindent {\sc \underline{Case II}: $0<|m(\lambda)|<m(\theta)$.} From the preceding discussion it is clear that only this case
gives us a truly complex-valued martingale
$(Z_n(\lambda))_{n\in\mn_0}$, the object we shall concentrate on
in what follows. In our analysis distinguishing the cases $p<2$
and $p\geq 2$ seems inevitable. To explain this point somewhat
informally we restrict our attention to the case $\theta=0$ and
note that the $L^p$-convergence of the martingale
$(Z_n(\lambda))_{n\in\mn_0}$ is regulated, among others, by the
asymptotic behavior of $\me (\sum_{j=1}^n \xi_j^2)^{p/2}$ as
$n\to\infty$ for $\xi_1$, $\xi_2,\ldots$ independent copies of the
random variable $|Z_1(\lambda)-1|$ with finite $p$th moment. If
$p\geq 2$, then $\me \xi_1^2<\infty$ and one expects that $$\me
\Bigg(\sum_{j=1}^n \xi_j^2 \Bigg)^{p/2}~\sim~ (\me
\xi_1^2)^{p/2}n^{p/2},\quad n\to\infty.$$ If $p\in (1,2)$ and $\me
\xi_1^2=\infty$ the last asymptotic relation is no longer true,
and one expects that in typical situations
\begin{equation}\label{impr}
0<{\lim\inf}_{n\to\infty} n^{-p/\alpha}\me \Bigg(\sum_{j=1}^n
\xi_j^2 \Bigg)^{p/2}\leq {\lim\sup}_{n\to\infty} n^{-p/\alpha}\me
\Bigg(\sum_{j=1}^n \xi_j^2 \Bigg)^{p/2}<\infty
\end{equation}
for some $\alpha\in (p,2)$. It seems that the $\alpha$ cannot be
expressed in terms of moments.

Before closing the section we recall that according to the
Kesten-Stigum theorem (see, for instance, Theorem 2.1 on p.~23 in
\cite{Asmussen+Hering:1983}) we have $\lim_{n\to\infty}\,Z_n(0)=0$
a.s.\ whenever $m(0)<\infty$ and $\me Z_1(0)\log^+Z_1(0)=\infty$.
However, by the Seneta-Heyde theorem (see, for instance, Theorem
5.1 on p.~83 and Corollary 5.3 on p.~85 in
\cite{Asmussen+Hering:1983}) there exists a positive slowly
varying function $\ell$ with $\lim_{t\to\infty}\,\ell(t)=\infty$
such that
$$ \lim_{n\to\infty}Z_n(0)\ell(m(0)^n)=\tilde{Z}_\infty(0)$$
for a
random variable $\tilde{Z}_\infty(0)$ which is positive with
positive probability.

\section{Main results}\label{results}

We are ready to state a criterion for the $L^p$-convergence, $p\in
(1,2)$. The cases $\theta=0$ and $\theta\neq 0$ are treated
separately in Theorems \ref{main3} and \ref{main4}, respectively.
\begin{thm}\label{main3}
Let $p\in (1,2)$, $\gamma\in\mr\backslash\{0\}$, $\lambda={\rm
i}\gamma$ and $0<|m(\lambda)|<m(0)<\infty$. Assume that
\begin{equation}\label{imp2}
\me |Z_1(\lambda)|^2<\infty
\end{equation}
or
\begin{equation}\label{imp1}
0<{\lim\inf}_{x\to\infty}x^\alpha\mmp\{|Z_1(\lambda)|>x\}\leq
{\lim\sup}_{x\to\infty}x^\alpha \mmp\{|Z_1(\lambda)|>x\}<\infty
\end{equation}
for some $\alpha\in (1,2)$. If either $\me
Z_1(0)\log^+Z_1(0)=\infty$ and
$$A:=\sum_{n\geq 0}\frac{1}{\ell(m(0)^n)^{p/\alpha}}=\infty,$$ where
$\ell$ is a slowly varying function appearing in the Seneta-Heyde
theorem, and we take $\alpha=2$ when condition \eqref{imp2} holds,
or $\me Z_1(0)\log^+Z_1(0)<\infty$, then the martingale
$(Z_n(\lambda))_{n\in\mn_0}$ converges in $L^p$ if, and only if,
\begin{equation}\label{palpha}
p<\alpha
\end{equation}
and
\begin{equation}\label{mom<2}
\frac{m(0)}{|m(\lambda)|^\alpha}<1.
\end{equation}
If $\me Z_1(0)\log^+Z_1(0)=\infty$ and $A<\infty$, then the
martingale $(Z_n(\lambda))_{n\in\mn_0}$ converges in $L^p$ if, and
only if, condition \eqref{palpha} holds and
\begin{equation}\label{momleq2}
\frac{m(0)}{|m(\lambda)|^\alpha}\leq 1.
\end{equation}
\end{thm}
\begin{rem}
A perusal of the proof of Theorem \ref{main3} reveals that
conditions \eqref{imp1} and \eqref{palpha} can be safely replaced
by the (seemingly) less restrictive condition \eqref{impr},
thereby extending the range of applicability of the result.
\end{rem}
\begin{rem}
Let us note that irrespective of the $x\log x$ condition
$Z_n(0)\ell(m(0)^n)$ converges a.s.\ to a random variable which is
positive with positive probability. Here, the slowly varying
function $\ell$ is identically one when $\me
Z_1(0)\log^+Z_1(0)<\infty$. In view of this we can reformulate
Theorem \ref{main3} in a more succinct form: under assumptions
\eqref{imp2} and \eqref{imp1} the martingale $(Z_n({\rm
i}\gamma))_{n\in\mn_0}$ converges in $L^p$, $p\in (1,2)$ if, and
only if, condition \eqref{palpha} holds and
$$\sum_{n\geq 0}\Big(\frac{m(0)}{|m({\rm
i}\gamma)|^\alpha}\Big)^n\frac{1}{\ell(m(0)^n)^{p/\alpha}}<\infty.$$
\end{rem}
\begin{thm}\label{main4}
Let $p\in (1,2)$, $\theta, \gamma\in\mr\backslash\{0\}$,
$\lambda=\theta+{\rm i}\gamma$ and
$0<|m(\lambda)|<m(\theta)<\infty$. Assume that conditions
\eqref{imp2} and \eqref{imp1} hold with the present $\lambda$, and
that the martingale $(Z_n(\alpha\theta))_{n\in\mn_0}$ is uniformly
integrable (we take $\alpha=2$ when condition \eqref{imp2} holds).
Then the martingale $(Z_n(\lambda))_{n\in\mn_0}$ converges in
$L^p$ if, and only if, condition \eqref{palpha} holds and
\begin{equation}\label{mom1<2}
\frac{m(\alpha\theta)}{|m(\lambda)|^\alpha}<1.
\end{equation}
\end{thm}
\begin{rem}
Necessary and sufficient conditions for the uniform integrability
of the Biggins martingale with real parameter were obtained in
increasing generality in \cite{Biggins:1977}, \cite{Lyons:1997}
and \cite{Alsmeyer+Iksanov:2009}. Simple sufficient conditions for
the uniform integrability of the martingale
$(Z_n(\alpha\theta))_{n\in\mn_0}$ are $\me
Z_1(\alpha\theta)\log^+Z_1(\alpha\theta)<\infty$ and
$-\alpha\theta\me \sum_{|u|=1}e^{-\alpha \theta S(u)}S(u)\in
[-\infty, m(\alpha\theta)\log m(\alpha\theta))$.
\end{rem}

Theorem \ref{main4} requires that the martingale
$(Z_n(\alpha\theta))_{n\in\mn_0}$ be uniformly integrable which is
an unpleasant feature. The problem is that it seems that the other
assumptions of Theorem \ref{main4} do not lead to any conclusions
concerning the asymptotics of $\me [Z_n(\alpha
\theta)]^{p/\alpha}$ as $n\to\infty$, when
$(Z_n(\alpha\theta))_{n\in\mn_0}$ is not uniformly integrable
martingale. Although in the latter case there are several results
(see \cite{Aidekon+Shi:2014, Biggins+Kyprianou:1997, Hu+Shi:2009})
concerning distributional convergence of $Z_n(\alpha \theta)a_n$
as $n\to\infty$ for appropriate constants $(a_n)$, the assumptions
imposed in the cited works are too restrictive for our purposes.
Fortunately, there is (at least) one exception arising in the case
$\theta=0$ which allowed us to provide a more complete result in
Theorem \ref{main3}.

Necessary and sufficient conditions given in Theorems \ref{main3}
and \ref{main4} like any other necessary and sufficient conditions
are of mainly theoretical interest. For applications easily
verifiable sufficient conditions are of greater use. Biggins in
Theorem 1 of \cite{Biggins:1992} shows that the conditions $\me [Z_1(\theta)]^\gamma<\infty$ for some $\gamma\in (1,2)$ and
$m(p\theta)/|m(\lambda)|^p<1$ for some $p\in (1,\gamma]$  %, for $p\in (1,2)$,
%the conditions $m(p\theta)/|m(\lambda)|^p<1$ and $\me
%[Z_1(\theta)]^p<\infty$
are sufficient for the $L^p$-convergence of
$(Z_n(\lambda)_{n\in\mn_0}$. Albeit looking differently
Proposition \ref{main1} given next is essentially equivalent to
the Biggins conditions, the %small
 improvement being that we use a
moment condition for $Z_1(\lambda)$ rather than for $Z_1(\theta)$.
%provides an improvement over the latter result.
\begin{assertion}\label{main1}
Let $p\in (1,2)$, $\gamma\in\mr\backslash\{0\}$,
$\lambda=\theta+{\rm i}\gamma$ and
$0<|m(\lambda)|<m(\theta)<\infty$. The conditions
\begin{equation}\label{inter12}
\me |Z_1(\lambda)|^r<\infty\quad \text{and}\quad
\frac{m(r\theta)}{|m(\lambda)|^r}<1
\end{equation}
for some $r\in [p,2]$ are sufficient for the $L^p$-convergence of
the martingale $(Z_n(\lambda))_{n\in\mn_0}$.
\end{assertion}

Now we formulate a criterion for the $L^p$-convergence, $p\geq 2$.
In the sequel we use  the standard notation
$$x\vee y=\max(x,y)   \quad \mbox{  and }  \quad   x\wedge
y=\min(x,y).$$
% will be used.
\begin{thm}\label{main2}
Let $p\geq 2$, $\theta,\gamma\in\mr$, $\gamma\neq 0$,
$\lambda=\theta+{\rm i}\gamma$ and
$0<|m(\lambda)|<m(\theta)<\infty$.   If $\theta\neq 0$, the
martingale $(Z_n(\lambda))_{n\in\mn_0}$ converges in $L^p$ if, and
only if,
\begin{equation}\label{mom_Z}
\me |Z_1(\lambda)|^p<\infty,
\end{equation}
\begin{equation}\label{mom_m}
\frac{m(2\theta)}{|m(\lambda)|^2}\vee\frac{m(p\theta)}{|m(\lambda)|^p}<1
\end{equation}
and, when $p>2$,
\begin{equation}\label{mom_Ztheta}
\me [Z_1(2\theta)]^{p/2}<\infty.
\end{equation}
If $\theta=0$, the martingale $(Z_n(\lambda))_{n\in\mn_0}$
converges in $L^p$ if, and only if, conditions \eqref{mom_Z} and
\eqref{mom_Ztheta} hold, and $$\frac{m(0)}{|m(\lambda)|^2}<1.$$
\end{thm}

\section{Proofs}\label{proofs}

We first formulate a version of the Burkholder inequality for complex-valued martingales. Although we think the result is known, we have not been able to locate it in the literature.
\begin{lemma}\label{burk}
Let $p>1$ and $(X_n)_{n\in\mn_0}$ be a complex-valued martingale
with $X_0=0$. Then the martingale $(X_n)_{n\in\mn_0}$ converges in
$L^p$ if, and only if, $\me (\sum_{n\geq
0}|X_{n+1}-X_n|^2)^{p/2}<\infty$. If one of these holds, then
\begin{equation}\label{burk_inequality}
c_p\me\Big(\sum_{n\geq 0}|X_{n+1}-X_n|^2\Big)^{p/2}\leq \me |X|^p\leq C_p\me \Big(\sum_{n\geq 0}|X_{n+1}-X_n|^2\Big)^{p/2}
\end{equation}
for appropriate positive and finite constants $c_p$ and $C_p$,
where $X$ is the $L^p$- limit of $(X_n)_{n\in\mn_0}$.
\end{lemma}
\begin{proof}
We only need to prove \eqref{burk_inequality}. According to Theorem 1 on
p.~414 in \cite{Chow+Teicher:1997} inequality
\eqref{burk_inequality} holds for real-valued martingales with
constants $c_p^\ast$ and $C_p^\ast$ in place of $c_p$ and $C_p$.
We shall deduce \eqref{burk_inequality} for complex-valued
martingales from the cited theorem and the fact that $({\rm
Re}\,X_n)_{n\in\mn_0}$ and $({\rm Im}\,X_n)_{n\in\mn_0}$ are
real-valued martingales. From the elementary inequalities
$$(2^{r-1} \wedge 1) (a^r +b^r)  \leq   (a+b)^r \leq (2^{r-1} \vee 1) (a^r +b^r),\quad a,b \geq 0,  \;  r>0$$
we obtain $$(2^{p/2-1}\wedge 1)(|{\rm Re}\,X|^p+|{\rm
Im}\,X|^p)\leq |X|^p\leq (2^{p/2-1}\vee 1)(|{\rm Re}\,X|^p+|{\rm
Im}\,X|^p).$$ Therefore,
\begin{eqnarray*}
&&\me |X|^p\\ &\leq& (2^{p/2-1}\vee 1)(\me |{\rm Re}\,X|^p+\me |{\rm Im}\,X|^p) \\
&\leq & (2^{p/2-1}\vee 1) C_p^\ast  \left(\me \left[  \Big(\sum_{n\geq 0}\big({\rm Re}\,(X_{n+1}-X_n)\big)^2\Big)^{p/2} \right]
                                                            + \me \left[ \Big(\sum_{n\geq 0}\big({\rm Im}\,(X_{n+1}-X_n)\big)^2\Big)^{p/2}  \right]  \right) \\
&\leq&  \frac{2^{p/2-1}\vee 1}{2^{p/2-1} \wedge 1}  \; C_p^\ast
\me \Big(\sum_{n\geq 0}\big|X_{n+1}-X_n\big|^2\Big)^{p/2}
\end{eqnarray*}
and
\begin{eqnarray*}
&&\me |X|^p\\&\geq& (2^{p/2-1}\wedge 1)(\me |{\rm Re}\,X|^p+\me |{\rm Im}\,X|^p)\\
&\geq& (2^{p/2-1}\wedge 1) c_p^\ast  \left( \me \left[ \Big(\sum_{n\geq 0}\big({\rm Re}\,(X_{n+1}-X_n)\big)^2\Big)^{p/2} \right]
                                                                   +\me \left[ \Big(\sum_{n\geq 0}\big({\rm Im}\,(X_{n+1}-X_n) \big)^2\Big)^{p/2} \right] \right) \\
&\geq& \frac{ 2^{p/2-1}\wedge 1} {2^{p/2-1}\vee 1}  \;   c_p^\ast  \me \Big(\sum_{n\geq 0}\big|X_{n+1}-X_n\big|^2\Big)^{p/2}.
\end{eqnarray*}
\end{proof}

In Lemma \ref{senhey} given next which is needed for the proof of Theorem \ref{main3} we use the notation introduced in the paragraph preceding Theorem \ref{main3}.
\begin{lemma}\label{senhey}
Let $r\in (0,1)$, $m(0)\in (1,\infty)$ and $\me
Z_1(0)\log^+Z_1(0)=\infty$. Then $\me
[\tilde{Z}_\infty(0)]^r<\infty$ and $$\me
[Z_n(0)]^r~\sim~\frac{\me
[\tilde{Z}_\infty(0)]^r}{\ell(m(0)^n)^r},\quad n\to\infty.$$
\end{lemma}
\begin{proof}
By Corollary 5.5 on p.~86 in \cite{Asmussen+Hering:1983}, the
function $x\mapsto \int_0^x \mmp\{\tilde{Z}_\infty(0)>y\}{\rm d}y$
slowly varies at $\infty$. This entails $\me
[\tilde{Z}_\infty(0)]^r<\infty$.

From Theorem 5.1 on p.~83 in \cite{Asmussen+Hering:1983} (and its
proof) and Corollary 5.3 on p.~85 in \cite{Asmussen+Hering:1983}
we know that $m(0)^{-n}\ell(m(0)^n)\sim h_n(s_0)$ as $n\to\infty$,
where $h_n(s)$ is the inverse function of $x\mapsto -\log\me
e^{-x\mm_n(\mr)}$ for $n\in\mn$ and $s_0$ is a small enough
positive number, and that $(\exp(-h_n(s_0)\mm_n(\mr)))_{n\in\mn}$
is a martingale with respect to the natural filtration which
converges a.s.\ and in mean as $n\to\infty$ to
$\exp(-\tilde{Z}_\infty(0))$. The first of these facts tells us
that it suffices to prove that
\begin{equation}\label{inter4}
\lim_{n\to\infty}\,\me [h_n(s_0)\mm_n(\mr)]^r=\me [\tilde{Z}_\infty(0)]^r.
\end{equation}
As a consequence of the second we infer that, for each $s\in (0,1)$, $(\exp(-sh_n(s_0)\mm_n(\mr)))_{n\in\mn}$ is a submartingale. In particular,
\begin{equation}\label{inter5}
1-\me e^{-s h_n(s_0)\mm_n(\mr)} \leq 1-\me e^{-s\tilde{Z}_\infty(0)},\quad s\in (0,1).
\end{equation}

To prove \eqref{inter4} we shall use the following formula which holds for any nonnegative random variable $X$ and $a\in (0,1)$:
\begin{equation}\label{moments}
\me X^a=\frac{a}{\Gamma(1-a)}\int_0^\infty s^{-a-1}(1-\me
e^{-sX}){\rm d}s,
\end{equation}
where $\Gamma(\cdot)$ is the gamma function. This equality follows
from $\me e^{-sX}=\mmp\{R >sX \}$ for $s\geq 0$, where $R$ is an
exponentially distributed random variable of unit mean which is
independent of $X$.

With the help of $\lim_{n\to\infty}\,\me
e^{-sh_n(s_0)\mm_n(\mr)}=\me e^{-s\tilde{Z}_\infty(0)}$ for all
$s\geq 0$, inequality \eqref{inter5} and  the fact that $1-\me
e^{-s h_n(s_0)\mm_n(\mr)} \leq 1$ for $s\geq 1,$ we obtain
\begin{eqnarray*}
\me [h_n(s_0)\mm_n(\mr)]^r&=&\frac{r}{\Gamma(1-r)}\int_0^\infty s^{-r-1}(1-\me e^{-sh_n(s_0)\mm_n(\mr)}){\rm d}s\\&\to&\frac{r}{\Gamma(1-r)}
\int_0^\infty s^{-r-1}(1-\me e^{-s\tilde{Z}_\infty(0)}){\rm d}s= \me [\tilde{Z}_\infty(0)]^r
\end{eqnarray*}
as $n\to\infty$ by Lebesgue's dominated convergence theorem.
\end{proof}

For any $u\in\V$ and $\lambda\in\mathcal{D}_{\neq 0}$, set
$$Z_1^{(u)}(\lambda):=\frac{1}{m(\lambda)} \sum_{|v|=1}e^{-\lambda
(S(uv)-S(u)}   \quad \mbox{ and } \quad  Y_u(\lambda):= \frac{e^{-\lambda
S(u)}}{m(\lambda)^{|u|}}.$$
Thus, $Z_1^{(u)}(\lambda)$ is the
analogue of $Z_1(\lambda)$, but based on the progeny of individual
$u$ rather than the progeny of the initial ancestor $\varnothing$.
Observe that, for the individuals $u$ with $|u|= n$ for some
$n\in\mn$, the $Y_u$ are $\mathcal{F}_n$--measurable, whereas the
$Z_1^{(u)}(\lambda)$ are independent of $\mathcal{F}_n$.
\begin{lemma}\label{mom_est}
Let $p\in (1,2)$, $\gamma\in\mr\backslash\{0\}$,
$\lambda=\theta+{\rm i}\gamma$ and
$0<|m(\lambda)|<m(\theta)<\infty$. Assume that \eqref{imp1} holds
for $\alpha\in (p,2)$ and, when $\theta\neq 0$, that
$m(\alpha\theta)<\infty$ and the martingale
$(Z_n(\alpha\theta))_{n\in\mn_0}$ is uniformly integrable. Then
there exist positive constants $c$ and $C$ such that for each
$n\in\mn$,
\begin{equation}\label{import}
c\Big(\frac{m(0)}{|m(\lambda)|^\alpha}\Big)^{np/\alpha}\me
[Z_n(0)]^{p/\alpha}\leq \me \Big(\sum_{|u|=n}|Y_u(\lambda)|^2
|Z_1^{(u)}(\lambda)-1|^2\Big)^{p/2}\leq
C\Big(\frac{m(0)}{|m(\lambda)|^\alpha}\Big)^{np/\alpha}\me
[Z_n(0)]^{p/\alpha}
\end{equation}
when $\theta=0$, and
\begin{equation}\label{import1}
c\Big(\frac{m(\alpha\theta)}{|m(\lambda)|^\alpha}\Big)^{np/\alpha}\leq
\me \Big(\sum_{|u|=n}|Y_u(\lambda)|^2
|Z_1^{(u)}(\lambda)-1|^2\Big)^{p/2}\leq
C\Big(\frac{m(\alpha\theta)}{|m(\lambda)|^\alpha}\Big)^{np/\alpha}
\end{equation}
when $\theta\neq 0$.
\end{lemma}
\begin{proof}
Denote by $\xi_1,\xi_2, \ldots$ independent random variables which
are distributed as $|Z_1(\lambda)-1|$ and independent of
$\mathcal{F}_n$. Further, let $\eta_1$, $\eta_2,\ldots$ be i.i.d.\
positive random variables with
$$\mmp\{\eta_1>x\}~\sim~ bx^{-\alpha},\quad x\to\infty$$ for some
$b>0$ and the same $\alpha$ as in \eqref{imp1}. It is clear that
$\varphi(s):=\me e^{-s\eta_1^2}$ satisfies
\begin{equation}\label{near zero}
-\log \varphi(s)~ \sim~ b\Gamma(1-\alpha/2)s^{\alpha/2},\quad s\to
0+.
\end{equation}
Let $\eta_{\alpha/2}$ be a positive $(\alpha/2)$-stable random
variable with the Laplace transform
$$\Psi(s):=\me \exp(-s\eta_{\alpha/2})=\exp(-b\Gamma(1-\alpha/2)s^{\alpha/2}),\quad
s\geq 0.$$

\noindent {\sc\underline{Case} $\theta=0$}. Set
$$\Psi_k(s):=\me \exp\bigg(\frac{\eta_1^2+\ldots+\eta_k^2}{k^{2/\alpha}}\bigg)=
[\varphi(sk^{-2/\alpha})]^k, \quad s\geq 0,~ k\in\mn.   $$ It is
easily seen that
\begin{equation}\label{conv}
\lim_{k\to\infty}\Psi_k(s)=\Psi(s),\quad s\geq 0.
\end{equation}
We intend to show that
\begin{equation}\label{inter17}
\lim_{k\to\infty}\me\Bigg(\frac{\eta_1^2+\ldots+\eta_k^2}{k^{2/\alpha}}\Bigg)^{p/2}=\me
[\eta_{\alpha/2}]^{p/2}<\infty.
\end{equation}
According to formula \eqref{moments} relation \eqref{inter17} is
equivalent to
\begin{equation}\label{limit3}
\lim_{k\to\infty}\int_0^\infty s^{-p/2-1}(1-\Psi_k(s)){\rm
d}s=\int_0^\infty s^{-p/2-1}(1-\Psi(s)){\rm d}s.
\end{equation}
With \eqref{conv} at hand we shall prove \eqref{limit3} with the
help of Lebesgue's dominated convergence theorem. In view of
\eqref{near zero}, for $s_0>0$ small enough there exists $r>0$
such that $-\log \varphi(s)\leq rs^{\alpha/2}$ whenever $s\in
[0,s_0]$. Hence, for such $s$
$$1-\Psi_k(s)\leq -\log \Psi_k(s)\leq rs^{\alpha/2},$$ and
$\int_0^{s_0}s^{-p/2-1+\alpha/2}{\rm
d}s=\frac{2}{\alpha-p}s_0^{(\alpha-p)/2}<\infty$ because
$p<\alpha$. For $s\geq s_0$ we use the crude estimate
$1-\Psi_n(s)\leq 1$ which suffices in view of $\int_{s_0}^\infty
s^{-p/2-1}{\rm d}s=(2/p)s_0^{-p/2}<\infty$. The proof of
\eqref{limit3} is complete.

As a consequence of \eqref{limit3} and \eqref{imp1} we obtain
\begin{equation*}
c\leq \me
\Big(\frac{\xi_1^2+\ldots+\xi_k^2}{k^{2/\alpha}}\Big)^{p/2}\leq C
\end{equation*}
for all $k\in\mn$ and appropriate $c, C>0$, whence
$$\me
\bigg(\sum_{j=1}^{\mm_n(\mr)}\xi_j^2 \bigg)^{p/2}=\me
[\mm_n(\mr)]^{p/\alpha} \me
\bigg(\bigg(\frac{\sum_{j=1}^{\mm_n(\mr)}\xi_j^2}
{\mm_n(\mr)^{2/\alpha}}\1_{\{\mm_n(\mr)\geq 1\}}
\bigg)^{p/2}\bigg|\mathcal{F}_n\bigg)\geq c \me
[\mm_n(\mr)]^{p/\alpha}.$$ Arguing similarly for the the upper
bound we arrive at
\begin{equation}\label{inter11}
c\me [\mm_n(\mr)]^{p/\alpha}\leq \me
\bigg(\sum_{j=1}^{\mm_n(\mr)}\xi_j^2 \bigg)^{p/2}\leq C\me
[\mm_n(\mr)]^{p/\alpha}
\end{equation}
which is equivalent to \eqref{import}.

\bigskip
\noindent  {\sc\underline{Case} $\theta\neq 0$}. Like in the
previous part of the proof, inequality \eqref{import1} follows if
we can show that
\begin{equation}\label{inter16}
\lim_{n\to\infty}\me\Bigg(\frac{\sum_{|u|=n}e^{-2\theta
S(u)}|Z_1^{(u)}(\lambda)-1|^2}{m(\alpha\theta)^{2n/\alpha}}\Bigg)^{p/2}=\me[
\eta_{\alpha/2} Z_\infty(\alpha\theta)^{2/\alpha}]^{p/2}<\infty
\end{equation}
assuming that $|Z_1(\lambda)-1|$ has the same distribution as
$\eta_1$. Here, $Z_\infty(\alpha\theta)$ is the a.s.\ and
$L_1$-limit of the uniformly integrable martingale
$(Z_n(\alpha\theta))_{n\in\mn_0}$. Furthermore,
$Z_\infty(\alpha\theta)$ is assumed independent of
$\eta_{\alpha/2}$. By \eqref{moments}, relation \eqref{inter16} is
equivalent to
\begin{equation}\label{limit7}
\lim_{n\to\infty}\int_0^\infty s^{-p/2-1}(1-\Phi_n(s)){\rm
d}s=\int_0^\infty s^{-p/2-1}(1-\Phi(s)){\rm d}s,
\end{equation}
where $$\Phi_n(s):=\me \exp\bigg(-s\frac{\sum_{|u|=n}e^{-2\theta
S(u)}|Z_1^{(u)}(\lambda)-1|^2}{m(\alpha\theta)^{2n/\alpha}}\Bigg),\quad
s\geq 0,~ n\in\mn$$ and $$\Phi(s):=\me \exp (-s
\eta_{\alpha/2}Z_\infty(\alpha\theta)^{2/\alpha})=\me
\exp(-b\Gamma(1-\alpha/2)s^{\alpha/2}Z_\infty(\alpha\theta)). $$
 % for $s\geq 0$.

By Theorem 3 in \cite{Biggins:1998},
$$\frac{\sup_{|u|=n}e^{-2\theta S(u)}}{m(\alpha\theta)^{2n/\alpha}}=\Big(\frac{\sup_{|u|=n}
e^{-\alpha \theta
S(u)}}{m(\alpha\theta)^n}\Big)^{2/\alpha}~\to~0\quad \text{a.s.}$$
as $n\to\infty$. This in combination with \eqref{near zero}
yields, for $s\geq 0$,
\begin{eqnarray*}
&&-\log \me \bigg(\exp \Bigg(-s\frac{\sum_{|u|=n}e^{-2\theta
S(u)}|Z_1^{(u)}(\lambda)-1|^2}{m(\alpha\theta)^{2n/\alpha}}\Bigg)\bigg|\mathcal{F}_n\bigg)\\&=&\sum_{|u|=n}-\log
\varphi\Big(s\frac{e^{-2\theta
S(u)}}{m(\alpha\theta)^{2n/\alpha}}\Big)\\&~\sim~&
b\Gamma(1-\alpha/2)s^{\alpha/2}Z_n(\alpha\theta)~\to~
b\Gamma(1-\alpha/2)s^{\alpha/2}Z_\infty(\alpha\theta)\quad
\text{a.s.}
\end{eqnarray*}
as $n\to\infty$, thereby proving that
$$\lim_{n\to\infty}\Phi_n(s)=\Phi(s),\quad s\geq 0.$$ To justify \eqref{limit7} we shall use Lebesgue's dominated convergence theorem.
As a consequence of \eqref{near zero}, given $s_0>0$ small enough
there exist positive constants $B_1$ and $B_2$ such that
\begin{equation*}
\frac{s^{\alpha/2}}{1-\varphi(s)}\leq B_1\quad\text{and}\quad
\frac{1-\varphi(sx)}{1-\varphi(s)}\leq B_2x^{\alpha/2}
\end{equation*}
whenever $s\in (0, s_0]$ and $sx\in (0, s_0]$. Therefore, for
$s\in (0, s_0]$ and $S(u)$ with $|u|=n$,
\begin{eqnarray*}
\frac{1-\varphi\Big(s\frac{e^{-2\theta
S(u)}}{m(\alpha\theta)^{2n/\alpha}}\Big)}{1-\varphi(s)}
&=&\frac{1-\varphi\Big(s\frac{e^{-2\theta
S(u)}}{m(\alpha\theta)^{2n/\alpha}}\Big)}{1-\varphi(s)}\1_{\{\frac{e^{-2\theta
S(u)}}{m(\alpha\theta)^{2n/\alpha}}\leq \frac{s_0}{s}\}}+
\frac{1-\varphi\Big(s\frac{e^{-2\theta
S(u)}}{m(\alpha\theta)^{2n/\alpha}}\Big)}{1-\varphi(s)}\1_{\{\frac{e^{-2\theta
S(u)}}{m(\alpha\theta)^{2n/\alpha}}>\frac{s_0}{s}\}}\\&\leq&
\Big(B_2+\frac{B_1}{s_0^{\alpha/2}}\Big)\frac{e^{-\alpha\theta
S(u)}}{m(\alpha\theta)^n}\quad \text{a.s.}
\end{eqnarray*}
This yields, for each $n\in\mn$ and $s\in (0, s_0]$,
\begin{eqnarray*}
1-\Phi_n(s)&=&\me
\bigg(1-\prod_{|u|=n}\varphi\bigg(s\frac{e^{-2\theta
S(u)}}{m(\alpha\theta)^{2n/\alpha}}\bigg)\bigg)\leq
(1-\varphi(s))\me
\sum_{|u|=n}\frac{1-\varphi\Big(s\frac{e^{-2\theta
S(u)}}{m(\alpha\theta)^{2n/\alpha}}\Big)}{1-\varphi(s)}\\&\leq&
\Big(B_2+\frac{B_1}{s_0^{\alpha/2}}\Big)(1-\varphi(s))\me
\sum_{|u|=n} \frac{e^{-\alpha\theta S(u)}}{m(\alpha\theta)^n}=
\Big(B_2+\frac{B_1}{s_0^{\alpha/2}}\Big)(1-\varphi(s)).
\end{eqnarray*}
The so obtained majorant is appropriate because
$$\int_0^{s_0}s^{-p/2-1}(1-\varphi(s)){\rm d}s<\infty$$ as a consequence of $\me [\eta_1]^p<\infty$ (recall that
$p<\alpha$). When $s>s_0$ the crude bound $1-\Phi_n(s)\leq 1$
suffices, for $\int_{s_0}^\infty s^{-p/2-1}{\rm d}s<\infty$. The
proof of Lemma \ref{mom_est} is complete.
\end{proof}

For the proof of Theorem \ref{main2} we shall need a version of
Lemma 3.3 in \cite{Alsmeyer_etal:2009}.
\begin{lemma}\label{AIPR}
Assume that $m(p\theta)\geq m(\theta)^p$ and $\me [Z_1(\theta)]^p<\infty$ for some $p>1$ and $\theta\in\mr\backslash\{0\}$. Then $$\me [Z_n(\theta)]^p=O\Big(n^c\Big(\frac{m(p\theta)}{m(\theta)^p}\Big)^n\Big),\quad n\to\infty$$ for a finite nonnegative  constant $c$ (explicitly known).
\end{lemma}

We are now ready to prove our main results.
\begin{proof}[Proof of Theorem \ref{main3}.]

\noindent {\sc Necessity of \eqref{palpha} and \eqref{mom<2} or
\eqref{momleq2}}. Set $R:=\sum_{n\geq
0}|Z_{n+1}(\lambda)-Z_n(\lambda)|^2$ and assume that
$(Z_n(\lambda))_{n\in\mn_0}$ converges in $L^p$, $p\in (1,2)$.
Then $\me R^{p/2}<\infty$ by Lemma \ref{burk}. In particular, this
entails $\me |Z_1(\lambda)|^p<\infty$ thereby showing the
necessity of \eqref{palpha}.

Let $(a_n)_{n\geq 0}$ be a sequence of positive numbers which
satisfies $a:=\sum_{n\geq 0}a_n<\infty$. Since the function
$x\mapsto x^{p/2}$ is concave on $[0,\infty)$ we infer
\begin{eqnarray*}
R^{p/2}&=&a^{p/2}\Big(\sum_{n\geq 0}(a_n/a)
(1/a_n)|Z_{n+1}(\lambda)-Z_n(\lambda)|^2\Big)^{p/2}\\&\geq&
a^{p/2}\sum_{n\geq 0}(a_n/a)
(1/a_n)^{p/2}|Z_{n+1}(\lambda)-Z_n(\lambda)|^p\\&=&a^{p/2-1}\sum_{n\geq
0}a_n^{1-p/2}|Z_{n+1}(\lambda)-Z_n(\lambda)|^p.
\end{eqnarray*}
Given $\F_n$, the random variable $Z_{n+1}(\lambda)-Z_n(\lambda)$,
being a weighted sum of i.i.d.\ complex-valued zero-mean random
variables, is the terminal value of a martingale. Hence, Lemma
\ref{burk} applies and gives
\begin{eqnarray}\label{inter2}
C_p\me\Big(\sum_{|u|=n}|Y_u(\lambda)|^2|Z_1^{(u)}(\lambda)-1|^2\Big)^{p/2}&\geq&
\me |Z_{n+1}(\lambda)-Z_n(\lambda)|^p\notag\\&\geq&
c_p\me\Big(\sum_{|u|=n}|Y_u(\lambda)|^2|Z_1^{(u)}(\lambda)-1|^2\Big)^{p/2}=:c_pA_n
\end{eqnarray}
(the left-hand inequality is not needed here and will be used
later).

Assume that condition \eqref{imp2} holds. %  that is,
Then $\alpha=2$ by our convention.
Using once again concavity of $x\mapsto x^{p/2}$ on $[0,\infty)$
we obtain
$$A_n\geq \me |Z_1(\lambda)-1|^p\me
\Big(\sum_{|u|=n}|Y_u(\lambda)|^2\Big)^{p/2}=\Big(\frac{m(0)}{|m(\lambda)|^2}\Big)^{np/2}\me
|Z_1(\lambda)-1|^p\me [Z_n(0)]^{p/2}$$ and thereupon
\begin{equation}\label{inter10}
\infty>\me R^{p/2}\geq a^{p/2-1}c_p \me |Z_1(\lambda)-1|^p
\sum_{n\geq 0}a_n^{1-p/2}
\Big(\frac{m(0)}{|m(\lambda)|^\alpha}\Big)^{np/\alpha}\me
[Z_n(0)]^{p/\alpha}.
\end{equation}

Assume now that condition \eqref{imp1} holds. % that is,
Then $\alpha\in
(p,2)$ (recall \eqref{palpha}). According to \eqref{import},
$$A_n\geq c\Big(\frac{m(0)}{|m(\lambda)|^\alpha}\Big)^{np/\alpha}\me
[Z_n(0)]^{p/\alpha}$$ and thereupon
$$\infty>\me R^{p/2}\geq a^{p/2-1}c c_p \sum_{n\geq
0}a_n^{1-p/2}
\Big(\frac{m(0)}{|m(\lambda)|^\alpha}\Big)^{np/\alpha}\me
[Z_n(0)]^{p/\alpha}.$$ Observe that the series on the right-hand
side is the same as in \eqref{inter10}. Further, we have to
consider two cases.

\noindent {\sc \underline{Case} $\me Z_1(0)\log^+ Z_1(0)<\infty$.}
According to the Kesten-Stigum theorem, already mentioned in
Section \ref{results}, $Z_n(0)$ converges a.s.\ and in mean as
$n\to\infty$ to a random variable $Z_\infty(0)$. Therefore,
$\lim_{n\to\infty}\,\me [Z_n(0)]^{p/\alpha} =\me
[Z_\infty(0)]^{p/\alpha}\in (0,\infty)$, and the necessity of
\eqref{mom<2} follows upon setting
\begin{equation}\label{seq}
a_n=\Big(\frac{m(0)}{|m(\lambda)|^\alpha}\Big)^n,\quad n\in\mn_0.
\end{equation}

\noindent {\sc \underline{Case} $\me Z_1(0)\log^+ Z_1(0)=\infty$.}
By Lemma \ref{senhey}, we have
$$\me [Z_n(0)]^{p/\alpha}~\sim~ \frac{\me [\tilde{Z}_\infty(0)]^{p/\alpha}}{\ell(m(0)^n)^{p/\alpha}},\quad n\to\infty$$
for some positive slowly varying $\ell$ with
$\lim_{t\to\infty}\,\ell(t)=\infty$. Assume that $\sum_{n\geq
0}\ell(m(0)^n)^{-p/\alpha}$ is a divergent series. Then choosing
$a_n$ as in \eqref{seq} we see that condition \eqref{mom<2} is
necessary. If the series $\sum_{n\geq 0}\ell(m(0)^n)^{-p/\alpha}$
converges then choosing any sequence $(a_n)_{n\in\mn_0}$ with the
property $\lim_{n\to\infty}e^{bn}a_n=\infty$ for any $b>0$, we
conclude that condition \eqref{momleq2} is necessary.

\noindent {\sc Sufficiency of \eqref{palpha} and \eqref{mom<2} or
\eqref{momleq2}}. By Lemma \ref{burk}, it suffices to show that
$\me R^{p/2}<\infty$. Using subadditivity of $x\mapsto x^{p/2}$ on
$[0,\infty)$ we obtain $$\me R^{p/2}\leq \sum_{n\geq 0}\me
|Z_{n+1}(\lambda)-Z_n(\lambda)|^p.$$ Further, in view of
\eqref{inter2} $$\me |Z_{n+1}(\lambda)-Z_n(\lambda)|^p\leq
C_p\me\Big(\sum_{|u|=n}|Y_u(\lambda)|^2|Z_1^{(u)}(\lambda)-1|^2\Big)^{p/2}.$$

Assume first that condition \eqref{imp2} holds,  % that is,
so that
$\alpha=2$. Using conditional Jensen's inequality yields
\begin{eqnarray*}
\me
\Big(\Big(\sum_{|u|=n}|Y_u(\lambda)|^2|Z_1^{(u)}(\lambda)-1|^2\Big)^{p/2}\Big|\mathcal{F}_n\Big)&\leq&
\me
\Big(\sum_{|u|=n}|Y_u(\lambda)|^2|Z_1^{(u)}(\lambda)-1|^2\Big|\mathcal{F}_n\Big)^{p/2}\\&=&\big[\me
|Z_1(\lambda)-1|^2\big]^{p/2}\Big(\frac{m(0)}{|m(\lambda)|^2}\Big)^{np/2}Z_n(0)^{p/2}~\text{a.s.},
\end{eqnarray*}
whence
$$\me R^{p/2}\leq C_p \big[\me
|Z_1(\lambda)-1|^2\big]^{p/2}\sum_{n\geq 0}
\Big(\frac{m(0)}{|m(\lambda)|^\alpha}\Big)^{np/\alpha}\me
[Z_n(0)]^{p/\alpha}.$$ Assume now that condition \eqref{imp1}
holds which together with \eqref{palpha} ensures that $\alpha\in
(p,2)$. In view of \eqref{import} $$A_n\leq C
\Big(\frac{m(0)}{|m(\lambda)|^\alpha}\Big)^{np/\alpha}\me
[Z_n(0)]^{p/\alpha}$$ which entails
$$\me R^{p/2}\leq C_pC \sum_{n\geq
0}\Big(\frac{m(0)}{|m(\lambda)|^\alpha}\Big)^{np/\alpha}\me
[Z_n(0)]^{p/\alpha}.$$ Arguing as in the proof of necessity we
conclude the following. If either $\me Z_1(0)\log^+ Z_1(0)=\infty$
and $A=\infty$, or $\me Z_1(0)\log^+ Z_1(0)<\infty$, then
condition \eqref{mom<2} is sufficient, whereas if\newline $\me
Z_1(0)\log^+ Z_1(0)=\infty$ and $A<\infty$, then condition
\eqref{momleq2} is sufficient. The proof of Theorem \ref{main3} is
complete.
\end{proof}

\begin{proof}[Proof of Theorem \ref{main4}.]
The proof is a simpler counterpart of the proof of Theorem
\ref{main3} which uses inequality \eqref{import1} rather than
\eqref{import}. We omit details.
\end{proof}

\begin{proof}[Proof of Proposition \ref{main1}.]
We have for $r$ satisfying \eqref{inter12}
\begin{eqnarray*}
\me R^{p/2}&\leq&
C_p\me\Big(\sum_{|u|=n}|Y_u(\lambda)|^2|Z_1^{(u)}(\lambda)-1|^2\Big)^{p/2}\\&\leq&
C_p
\me\Big(\sum_{|u|=n}|Y_u(\lambda)|^r|Z_1^{(u)}(\lambda)-1|^r\Big)^{p/r}\leq
C_p[\me |Z_1(\lambda)-1|^r]^{p/r}\sum_{n\geq
0}\Big(\frac{m(r\theta)}{|m(\lambda)|^r}\Big)^{np/r}<\infty
\end{eqnarray*}
which proves the result in view of Lemma \ref{burk}. The first
inequality was obtained in the proof of sufficiency in Theorem
\ref{main3}. The second and third are consequences of
subadditivity of $x\mapsto x^{r/2}$ and Jensen's inequality,
respectively.
\end{proof}

\begin{proof}[Proof of Theorem \ref{main2}.]
{\sc Necessity of \eqref{mom_Z}, \eqref{mom_m} and
\eqref{mom_Ztheta}}. Assume that $(Z_n(\lambda))_{n\in\mn_0}$
converges in $L^p$ and recall the notation $R=\sum_{n\geq
0}|Z_{n+1}(\lambda)-Z_n(\lambda)|^2$.  Then $\me R^{p/2}<\infty$
by Lemma \ref{burk}. Recalling that $p\geq 2$ and using
superadditivity of $x\mapsto x^{p/2}$ on $[0,\infty)$ we further
infer
\begin{equation}\label{inter}
\sum_{n\geq 0}\me |Z_{n+1}(\lambda)-Z_n(\lambda)|^p\leq \me R^{p/2}<\infty.
\end{equation}
On the one hand, we obtain for $A_n$ defined in \eqref{inter2},
$$A_n\geq\me
\sum_{|u|=n}|Y_u(\lambda)|^p|Z_1^{(u)}(\lambda)-1|^p=\me
|Z_1(\lambda)-1|^p\Big(\frac{m(p\theta)}{|m(\lambda)|^p}\Big)^n$$
having utilized the aforementioned superadditivity. In view of
\eqref{inter} this proves the necessity of \eqref{mom_Z} for
$p\geq 2$ and $m(p\theta)<|m(\lambda)|^p$. On the other hand, we
conclude that
\begin{eqnarray*}
A_n&\geq&\me \Big[\me \Big( \sum_{|u|=n}|Y_u(\lambda)|^2|Z_1^{(u)}(\lambda)-1|^2\Big|\mathcal{F}_n\Big)^{p/2}\Big]=\Big(\me |Z_1(\lambda)-1|^2\Big)^{p/2}\me \Big(\sum_{|u|=n}|Y_u(\lambda)|^2\Big)^{p/2}\\&\geq& \bigg(\me |Z_1(\lambda)-1|^2\Big(\frac{m(2\theta)}{|m(\lambda)|^2}\Big)^n\bigg)^{p/2},
\end{eqnarray*}
where the first and second inequalities are consequences of conditional and usual Jensen's inequality, respectively. This proves the necessity of $m(2\theta)<|m(\lambda)|^2$. Using the last chain of inequalities with $n=1$ we observe that $$\me \Big(\sum_{|u|=1}|Y_u(\lambda)|^2\Big)^{p/2}=\frac{1}{|m(\lambda)|^p}\me \Big(\sum_{|u|=1}e^{-2\theta S(u)}\Big)^{p/2}<\infty$$ which in combination with the already checked finiteness of $m(2\theta)$ proves the necessity of \eqref{mom_Ztheta}. Finally, if $\theta=0$, then conditions $m(0)>1$ and $m(0)<|m(\lambda)|^2$ imply that $|m(\lambda)|>1$. Therefore, $m(0)<|m(\lambda)|^p$ is a consequence of $m(0)<|m(\lambda)|^2$.

\medskip
\noindent {\sc Sufficiency of \eqref{mom_Z}, \eqref{mom_m} and
\eqref{mom_Ztheta}}. By Lemma \ref{burk}, it suffices to check
that $\me R^{p/2}<\infty$. Using the triangle inequality in
$L_{p/2}$ yields $$\me R^{p/2}\leq \bigg(\sum_{n\geq 0}\big[\me
|Z_{n+1}(\lambda)-Z_n(\lambda)|^p\big]^{2/p}\bigg)^{p/2}.$$ To
show that the right-hand side is finite, we write
\begin{eqnarray}
C_p^{-1}\me |Z_{n+1}(\lambda)-Z_n(\lambda)|^p&\leq& \me\bigg(\sum_{|u|=n}|Y_u(\lambda)|^2|Z_1^{(u)}(\lambda)-1|^2\bigg)^{p/2}\notag\\&=&\me\bigg(\sum_{|v|=n}|Y_v(\lambda)|^2 \sum_{|u|=n}\frac{|Y_u(\lambda)|^2}{\sum_{|v|=n}|Y_v(\lambda)|^2}|Z_1^{(u)}(\lambda)-1|^2\bigg)^{p/2}\notag\\&\leq& \me |Z_1(\lambda)-1|^p\me \bigg(\sum_{|u|=n}|Y_u(\lambda)|^2\bigg)^{p/2}\notag\\&=&\me |Z_1(\lambda)-1|^p \bigg(\frac{m(2\theta)}{|m(\lambda)|^2}\bigg)^{np/2}\me [Z_n(2\theta)]^{p/2}.\label{inter3}
\end{eqnarray}
We have used \eqref{inter2} for the first inequality and convexity of $x\mapsto x^{p/2}$ on $[0,\infty)$ for the second. Now we have to analyze the asymptotic behavior of $\me [Z_n(2\theta)]^{p/2}$ as $n\to\infty$. While doing so, distinguishing % several
two cases  seems inevitable.

\noindent {\sc \underline{Case} $p=2$}. The right-hand side of
\eqref{inter3} is equal to $\me |Z_1(\lambda)-1|^2
\Big(\frac{m(2\theta)}{|m(\lambda)|^2}\Big)^n$. Therefore,
conditions $\me |Z_1(\lambda)|^2<\infty$ and
$m(2\theta)<|m(\lambda)|^2$ ensure $\me R <\infty$.

\noindent {\sc \underline{Case} $p>2$}.

\noindent {\sc Subcase $m(p\theta)<m(2\theta)^{p/2}$,
$\theta\in\mr$}. In view of the present assumption on $m$ and
\eqref{mom_Ztheta} we have $\sup_{n\geq 0}\,\me
[Z_n(2\theta)]^{p/2}<\infty$ by Proposition \ref{conv_Lp}. Hence,
the right-hand side of \eqref{inter3} is
$O\Big(\Big(\frac{m(2\theta)}{|m(\lambda)|^2}\Big)^{np/2}\Big)$.
This in combination with $m(2\theta)<|m(\lambda)|^2$ proves $\me
R^{p/2}<\infty$. If $\theta=0$, this completes the proof of
sufficiency because the complementary case considered below which
reads $m(0)\geq m(0)^{p/2}$ is impossible in view of $m(0)\in
(1,\infty)$.

\noindent {\sc Subcase $m(p\theta)\geq m(2\theta)^{p/2}$,
$\theta\in\mr\backslash\{0\}$}. In view of the present assumption
on $m$ and \eqref{mom_Ztheta} we can apply Lemma \ref{AIPR} with
$2\theta$ and $p/2$ replacing $\theta$ and $p$ to obtain $$\me
[Z_n(2\theta)]^{p/2}=O\Big(n^c\Big(\frac{m(p\theta)}{m(2\theta)^{p/2}}\Big)^n
\Big),\quad n\to\infty$$ for appropriate finite constant $c$.
Hence, the right-hand side of \eqref{inter3} is
$O\Big(n^c\Big(\frac{m(p\theta)}{|m(\lambda)|^p}\Big)^n\Big)$
which proves $\me R^{p/2}<\infty$ because
$m(p\theta)<|m(\lambda)|^p$.

The proof of Theorem \ref{main2} is complete.
\end{proof}

\bigskip

\noindent {\bf Acknowledgement}. A part of this work was done
while A. Iksanov was visiting Vannes in October 2017. He
gratefully acknowledges hospitality and the financial support by
Universit\'{e} de Bretagne-Sud.
The work has been partially supported by the National Natural Science Foundation of China (Grants nos.  11601019, 11731012, 11571052), by  the Natural Science Foundation of Hunan Province of China (Grant No. 2017JJ2271), and  by the Centre Henri Lebesgue (CHL,
ANR-11-LABX-0020-01, France).

%The authors thank an anonymous referee for careful reading and
%pointing out several authors' oversights.

\end{document}